\documentclass[11pt]{article}
\usepackage[utf8]{inputenc}
\usepackage{graphicx}
\usepackage{amsmath}
\usepackage{color}
\usepackage{appendix}
\usepackage{blkarray}
\usepackage{fullpage}
\usepackage{mathtools}
\usepackage{commath}
\usepackage{csquotes}
\usepackage{amsfonts}
\usepackage{amssymb}
\usepackage{amsthm}
\newtheorem{thm}{Theorem}[section]
\newtheorem{lem}{Lemma}[section]

\newtheorem{dfn}{Definition}[section]
\newtheorem{cor}{Corollary}[section]

\DeclareMathOperator{\diam}{diam}

\DeclareMathOperator{\Spec}{Spec}

\title{Extremal problems for the eccentricity matrices of complements of trees}
\author{Iswar Mahato \thanks{Department of Mathematics, Indian Institute of Technology Kharagpur, Kharagpur 721302, India. Email: iswarmahato02@gmail.com, iswarmahato02@iitkgp.ac.in}\  \and M. Rajesh Kannan\thanks{Department of Mathematics, Indian Institute of Technology Hyderabad, Hyderabad 502285, India. Email: rajeshkannan1.m@gmail.com, rajeshkannan@math.iith.ac.in }}
\date{\today}
\begin{document}
\maketitle

\begin{abstract}
	The eccentricity matrix of a connected graph $G$, denoted by $\mathcal{E}(G)$, is obtained from the distance matrix of $G$ by keeping the largest nonzero entries in each row and each column, and leaving zeros in the remaining ones. The $\mathcal{E}$-eigenvalues of $G$ are the eigenvalues of $\mathcal{E}(G)$, in which the largest one is the $\mathcal{E}$-spectral radius of $G$. The $\mathcal{E}$-energy of $G$ is the sum of the absolute values of all $\mathcal{E}$-eigenvalues of $G$. In this article, we study some of the extremal problems for eccentricity matrices of complements of trees and characterize the extremal graphs. First, we determine the unique tree whose complement has minimum (respectively, maximum) $\mathcal{E}$-spectral radius among the complements of trees. Then, we prove that the $\mathcal{E}$-eigenvalues of the complement of a tree are symmetric about the origin. As a consequence of these results, we characterize the trees whose complement has minimum (respectively, maximum) least $\mathcal{E}$-eigenvalues among the complements of trees. Finally, we discuss the extremal problems for the second largest $\mathcal{E}$-eigenvalue and the $\mathcal{E}$-energy of complements of trees and characterize the extremal graphs. As an application, we obtain a Nordhaus-Gaddum type lower bounds for the second largest $\mathcal{E}$-eigenvalue and $\mathcal{E}$-energy of a tree and its complement. 
\end{abstract}



\section{Introduction}\label{sec1}
Throughout the paper, we consider finite, simple, and connected graphs. Let $G=(V(G),E(G))$ be a graph with the vertex set $V(G)=\{v_1,v_2,\hdots,v_n\}$ and the edge set $E(G)=\{e_1,e_2,\hdots,e_m\}$. The number of vertices in $G$ is the \textit{order} of $G$. If two vertices $u$ and $v$ in $G$ are adjacent, then we write $u\sim v$, otherwise $u\nsim v$. The \textit{adjacency matrix} $A(G)$ of $G$ is the $n \times n$ matrix with its rows and columns indexed by the vertices of $G$, and the entries are defined as
$$A(G)_{uv}=
\begin{cases}
	\text{$1$} & \quad\text{if $u\sim v$,}\\
	\text{0} & \quad\text{otherwise.}
\end{cases}$$
Let $\lambda_1(G)\geq \lambda_2(G)\geq \hdots \geq \lambda_n(G)$ be the eigenvalues of the adjacency matrix $A(G)$ of $G$. The largest eigenvalue of $A(G)$ is the \textit{spectral radius} of $G$. The \textit{energy (or the $A$-energy)} of $G$ is defined as $E_A(G)=\sum_{i=1}^n \abs{\lambda_i(G)}$. The \textit{distance} between the vertices $u$ and $v$ in $G$, denoted by $d_G(u,v)$, is the length of a shortest path between them in $G$, and define $d_G(u,u) =0$  for all $u \in V(G)$. The \textit{distance matrix} $D(G)$ of $G$ is the $n \times n$ matrix whose rows and columns are indexed by the vertices of $G$ and the $(u,v)$-th entry is equal to $d_G(u,v)$. Let $N_G(v)$ denote the collection of vertices that are adjacent to the vertex $v$ in $G$, and $N_G(v)$ is called the \textit{neighborhood} of $v $ in $G$. The \textit{eccentricity} $e_G(v)$ of a vertex $v\in V(G)$ is the maximum distance from $v$ to all other vertices of $G$. The maximum eccentricity of all vertices of $G$ is the \textit{diameter} of $G$, which is denoted by $\diam(G)$. The \textit{diametrical path} is a path whose length is equal to the diameter of $G$.

The \textit{eccentricity matrix} of a graph $G$ on $n$ vertices, denoted by $\mathcal{E}(G)$, is the $n\times n$ matrix whose rows and columns are indexed by the vertices of $G$, and the entries are defined as
$$\mathcal{E}(G)_{uv}=
\begin{cases}
	\text{$d_G(u,v)$} & \quad\text{if $d_G(u,v)=\min\{e_G(u),e_G(v)\}$,}\\
	\text{0} & \quad\text{otherwise.}
\end{cases}$$
The eccentricity matrix $\mathcal{E}(G)$ of a graph $G$ is a real symmetric matrix, and hence all of its eigenvalues are real. The eigenvalues of $\mathcal{E}(G)$ are the $\mathcal{E}$-eigenvalues of $G$, in which the largest one is the $\mathcal{E}$-spectral radius of $G$. The set of all $\mathcal{E}$-eigenvalues of $G$ is the $\mathcal{E}$-spectrum of $G$. If $\xi_1>\xi_2>\hdots >\xi_k$ are the distinct $\mathcal{E}$-eigenvalues of $G$, then we write the $\mathcal{E}$-spectrum of $G$ as
\[ \Spec_{\varepsilon}(G)=
\left\{ {\begin{array}{cccc}
		\xi_1 & \xi_2  &\hdots & \xi_k\\
		m_1& m_2& \hdots &m_k\\
\end{array} } \right\},
\]
where $m_i$ is the  multiplicity of $\xi_i$ for $i=1,2,\hdots,k$. 

Spectral extremal problems are one of the interesting problems in spectral graph theory. Recently, the extremal problems for eccentricity matrices of graphs have gained significant importance and attracted the attention of researchers. In \cite{wei2020solutions}, Wei et al. considered the extremal problem for $\mathcal{E}$-spectral radius of trees and determined the trees with minimum $\mathcal{E}$-spectral radius among all trees on $n$ vertices. Also, they characterized the trees with minimum $\mathcal{E}$-spectral radius among the trees with a given diameter. In \cite{mahato2021spectral}, the authors studied the minimal problem for $\mathcal{E}$-spectral radius of graphs with a given diameter and characterized the extreme graphs. Moreover, they identified the unique bipartite graph with minimum $\mathcal{E}$-spectral radius. Wang et al. \cite{wang2020boiling} characterized the graphs with minimum and second minimum $\mathcal{E}$-spectral radius as well as the graphs with maximum least and second least $\mathcal{E}$-eigenvalues. Recently, He and Lu \cite{he2022largest} considered the maximal problem for the $\mathcal{E}$-spectral radius of trees with the fixed odd diameter and determined the extremal trees. Wei et al. \cite{wei2022characterizing}  characterized the trees with second minimum $\mathcal{E}$-spectral radius and identified the trees with the small matching number having the minimum $\mathcal{E}$-spectral radius. Very Recently, Wei and Li \cite{wei2022eccentricity} studied the relationship between the majorization and $\mathcal{E}$-spectral radius of complete multipartite graphs and determined the extremal complete multipartite graphs with minimum and maximum $\mathcal{E}$-spectral radius. Mahato and Kannan \cite{mahato2022minimizers} considered the extremal problem for the second largest $\mathcal{E}$-eigenvalue of trees and determined the unique tree with minimum second largest $\mathcal{E}$-eigenvalue among all trees on $n$ vertices other than the star. For more advances on the eccentricity matrices of graphs, we refer to \cite{lei2022spectral,lei2021eigenvalues,mahato2020spectra,mahato2022eccentricity,mahato2022inertia,patel2021energy,wang2022spectraldetermination,wang2020spectral}.


The \textit{eccentricity energy (or the $\mathcal{E}$-energy)} of a graph $G$ is defined \cite{wang2019graph} as 
$$E_{\mathcal{E}}(G)=\sum_{i=1}^n \abs{\xi_i},$$ 
where $\xi_1,\xi_2,\hdots,\xi_n$ are the  $\mathcal{E}$-eigenvalues of $G$. Recently, many researchers focused on the eccentricity energy of graphs. In \cite{wang2019graph}, Wang et al. studied the $\mathcal{E}$-energy of graphs and obtained some bounds for the $\mathcal{E}$-energy of graphs and determined the corresponding extremal graphs. Lei et al. \cite{lei2021eigenvalues} obtained an upper bound for the $\mathcal{E}$-energy of graphs and characterized the extremal graphs. Very recently, Mahato and Kannan \cite{mahato2022minimizers} studied the minimization problem for the $\mathcal{E}$-energy of trees and characterized the trees with minimum $\mathcal{E}$-energy among all trees on $n$ vertices. For more details about the $\mathcal{E}$-energy of graphs, we refer to \cite{mahato2021spectral,mahato2022eccentricity,patel2021energy}.

Motivated by the above-mentioned works, in this article, we study the extremal problems for eccentricity matrices of complements of trees and characterize the extremal graphs for the $\mathcal{E}$-spectral radius, second largest $\mathcal{E}$-eigenvalue, least $\mathcal{E}$-eigenvalue and $\mathcal{E}$-energy among the complements of trees. The extremal problems for the complements of trees with respect to the other graph matrices have been studied in \cite{Fan-least-comple,Li-leastsign-comple,Li-Distsign-comple,Lin-Dist-comple}. For a tree $T$, let $T^c$ be the complement of $T$. Throughout the paper, we always assume that $T^c$ is connected; hence, $T$ is not isomorphic to the star graph. Let $\mathcal{T}_n^c$ denote the complements of all trees on $n$ vertices. Note that if $T$ is a tree with $\diam(T)>3$, then $\mathcal{E}(T^c)=2A(T)$. If $\diam(T)=3$, then $\mathcal{E}(T^c)\geq 2A(T)$ entrywise. Since $A(T)$ is irreducible, therefore $\mathcal{E}(T^c)$ is also irreducible.

The article is organized as follows: In section $2$, we collect needed notations and  some preliminary results. In section $3$, we characterize the extremal graphs with the minimum and maximum $\mathcal{E}$-spectral radius among the complements of trees. As a consequence, we determine the unique graphs with the minimum and maximum least $\mathcal{E}$-eigenvalues among the complements of trees. We discuss the extremal problems for the second largest $\mathcal{E}$-eigenvalue and the $\mathcal{E}$-energy of complements of trees in section $4$ and section $5$, respectively.

\section{Preliminaries}
In this section, we introduce some notations and collect some preliminary results, which will be used in the subsequent sections.
First, we define the quotient matrix and equitable partition.
\begin{dfn}[Equitable partition \cite{brou-haem-book}] Let $A$ be a real symmetric matrix whose rows and columns are indexed by $X=\{1,2,\hdots,n\}$. Let $\Pi=\{X_1,X_2,\hdots,X_m\}$ be a partition of $X$. The characteristic matrix $C$ is the $n\times m$ matrix whose $j$-th column is the characteristic vector of $X_j$ $(j=1,2,\hdots,m)$. Let $A$ be partitioned according to $\Pi$ as \[A=\left[ {\begin{array}{cccc}
			A_{11} & A_{12} &\hdots & A_{1m}\\
			A_{21} & A_{22} &\hdots & A_{2m}\\
			\vdots &\hdots & \ddots & \vdots\\
			A_{m1} & A_{m2}& \hdots &A_{mm}\\
	\end{array} } \right],\]
	where $A_{ij}$ denotes the submatrix (block) of $A$ formed by rows in $X_i$ and the columns in $X_j$. If $q_{ij}$ denotes the average row sum of $A_{ij}$, then the matrix $Q=(q_{ij})$ is called the quotient matrix of $A$. If the row sum of each block $A_{ij}$ is a constant, then the partition $\Pi$ is called \emph{equitable partition}.
\end{dfn}

In the following theorem, we state a well-known result about the spectrum of a quotient matrix corresponding to an equitable partition.

\begin{thm}[\cite{brou-haem-book}]\label{quo-spec}
	Let $Q$ be a quotient matrix of any square matrix $A$ corresponding to an equitable partition. Then the spectrum of $A$ contains the spectrum of $Q$.
\end{thm}

\begin{figure}[h!]
	\centering
	\includegraphics[scale= 0.90]{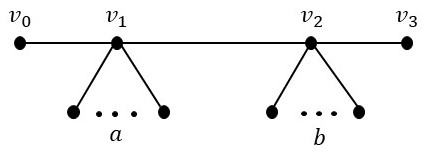}
	\caption{The tree $T_{n,3}^{a,b}$, where $a+b=n-4$.}
	\label{fig1}
\end{figure}

\begin{figure}[h!]
	\centering
	\includegraphics[scale= 0.80]{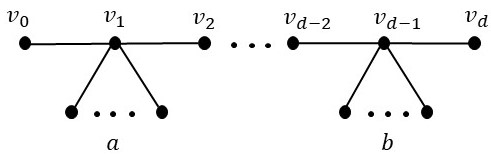}
	\caption{The tree $D_{n,d}^{a,b}$, where $a+b=n-d-1$ and $d\geq 4$.}
	\label{fig2}
\end{figure}

Let $K_n$, $P_n$, and $K_{1,n-1}$ denote the complete graph, the path, and the star on $n$ vertices, respectively. For $d=3$, let $T_{n,d}^{a,b}$ be the tree obtained from $P_4=v_0v_1v_2v_3$ by attaching $a$ pendant vertices to $v_1$ and $b$ pendent vertices to $v_2$, where $a+b=n-4$ and $b\geq a\geq 0$. For $d\geq 4$, let $D_{n,d}^{a,b}$ be the tree obtained from $P_{d+1}=v_0v_1v_2\hdots v_d$ by attaching $a$ pendant vertices to $v_1$ and $b$ pendent vertices to $v_{d-1}$, where $a+b=n-d-1$ and $b\geq a\geq 0$. The trees $T_{n,3}^{a,b}$ and $D_{n,d}^{a,b}$ are depicted in Figure \ref{fig1} and Figure \ref{fig2}, respectively. For a real number $x$, let $\lfloor x \rfloor$ denote the greatest integer less than or equal to $x$, and $\lceil x \rceil$ denote the least integer greater than or equal to $x$. In  the following lemma, we compute the $\mathcal{E}$-spectrum of the complement of the tree $T_{n,3}^{a,b}$, where $a+b=n-4$ and $b\geq a\geq 0$.

\begin{lem}\label{spec-tpq}
	For $b\geq a\geq 0$ with $a+b=n-4$, the $\mathcal{E}$-spectrum of $(T_{n,3}^{a,b})^c$ is given by
	\[ \Spec_{\mathcal{E}}\big((T_{n,3}^{a,b})^c\big)=
	\left\{ {\begin{array}{ccc}
			-\sqrt{\frac{4n+1\pm \sqrt{(4n+1)^2-64(a+1)(b+1)}}{2}} & 0  &\sqrt{\frac{4n+1\pm \sqrt{(4n+1)^2-64(a+1)(b+1)}}{2}} \\
			1 & n-4 & 1 \\
	\end{array} } \right\}.\]
\end{lem}
\begin{proof}
	Let $T_{n,3}^{a,b}$ be the tree obtained from $P_4=v_0v_1v_2v_3$ by attaching $a$ pendant vertices $u_1,\hdots,u_a$ to $v_1$ and $b$ pendent vertices $w_1,\hdots,w_b$ to $v_2$, where $a+b=n-4$ and $b\geq a\geq 0$. Then the eccentricity matrix of $(T_{n,3}^{a,b})^c$ is given by
	\[ \mathcal{E}\big((T_{n,3}^{a,b})^c\big)=
	\begin{blockarray}{ccccccccccc}
		& u_1 & \hdots & u_a & v_0 & v_1 & v_2 & v_3 & w_1 & \hdots & w_b \\
		\begin{block}{c(cccccccccc)}
			u_1 & 0 & \hdots & 0 & 0 & 2 & 0 & 0 & 0 & \hdots & 0\\
			\vdots & \vdots & \ddots & \vdots & \vdots & \vdots & \vdots & \vdots & \vdots & \ddots & \vdots \\
			u_a & 0 & \hdots & 0 & 0 & 2 & 0 & 0 & 0 & \hdots & 0\\
			v_0 & 0 & \hdots & 0 & 0 & 2 & 0 & 0 & 0 & \hdots & 0\\
			v_1 & 2 & \hdots & 2 & 2 & 0 & 3 & 0 & 0 & \hdots & 0\\
			v_2 & 0 & \hdots & 0 & 0 & 3 & 0 & 2 & 2 & \hdots & 2\\
			v_3 & 0 & \hdots & 0 & 0 & 0 & 2 & 0 & 0 & \hdots & 0\\
			w_1 & 0 & \hdots & 0 & 0 & 0 & 2 & 0 & 0 & \hdots & 0\\
			\vdots & \vdots & \ddots & \vdots & \vdots & \vdots & \vdots & \vdots & \vdots & \ddots & \vdots \\
			w_b & 0 & \hdots & 0 & 0 & 0 & 2 & 0 & 0 & \hdots & 0\\
		\end{block}
	\end{blockarray}~.\]
	
	It is easy to see that the rank of $\mathcal{E}\big((T_{n,3}^{a,b})^c\big)$ is $4$. Therefore, $0$ is an eigenvalue of $\mathcal{E}\big((T_{n,3}^{a,b})^c\big)$ with multiplicity $n-4$. 
	
	If $U=\{u_1,\hdots,u_a,v_0\}$ and $W=\{v_3,w_1,\hdots,w_b\}$, then $\Pi_1=U \cup \{v_1\}\cup \{v_2\}\cup W$ is an equitable partition of $\mathcal{E}\big((T_{n,3}^{a,b})^c\big)$ with the quotient matrix
	
	\[Q_1=
	\left( {\begin{array}{cccc}
			0 & 2 & 0 & 0 \\
			2(a+1) & 0 & 3 & 0 \\
			0 & 3 & 0 & 2(b+1) \\
			0 & 0 & 2 & 0 \\
	\end{array} } \right).\]
	By a direct calculation, the eigenvalues of $Q_1$ are
	\begin{align*}
		& \pm \sqrt{\frac{4a+4b+17+\sqrt{(4a+4b+17)^2-64(a+1)(b+1)}}{2}} \quad \text{and} \\
		& \pm \sqrt{\frac{4a+4b+17-\sqrt{(4a+4b+17)^2-64(a+1)(b+1)}}{2}}.    
	\end{align*}
	Now, the proof follows from Theorem \ref{quo-spec} by substituting $a+b=n-4$.
\end{proof}

In the following theorem, we find the adjacency energy of the path graph on $n$ vertices. We include proof of this result for the sake of completeness.

\begin{thm}\label{energy=path}
	The energy of a path $P_n$ on $n$ vertices is given by 
	\begin{align*}
		E_A(P_n)=\begin{cases}
			\text{$ 2\Big(\cot\big({\frac{\pi}{2(n+1)}}\big)-1\Big)$} & \quad\text{if $n$ is odd,}\\
			\text{$2\Big(\csc\big({\frac{\pi}{2(n+1)}}\big)-1\Big)$} & \quad\text{if $n$ is even.}
		\end{cases} 
	\end{align*}
\end{thm}
\begin{proof}
	We know that the eigenvalues of $P_n$ are $2\cos{\frac{\pi k}{n+1}}$, $k=1,2,\hdots,n$. By using the formula $\cos{x}+\cos{2x}+\hdots+\cos{nx}=\sin \big(\frac{nx}{2}\big)\csc \big(\frac{x}{2}\big)\cos \big(\frac{(n+1)x}{2}\big)$(for a proof of this identity, we refer to \cite{knapp2009sines}), we have
	\begin{align*}
		E_A(P_n)=&~\begin{cases}
			\text{$ 4\Big(\cos\big({\frac{\pi}{n+1}}\big)+\cos\big({\frac{2\pi}{n+1}}\big)+\hdots+\cos\big({\frac{(n-1)\pi}{2(n+1)}}\big)\Big)$} & \quad\text{if $n$ is odd,}\\
			\text{$4\Big(\cos\big({\frac{\pi}{n+1}}\big)+\cos\big({\frac{2\pi}{n+1}}\big)+\hdots+\cos\big({\frac{n\pi}{2(n+1)}}\big)\Big)$} & \quad\text{if $n$ is even;}
		\end{cases} \\
		=&~\begin{cases}
			\text{$ 4\sin\big({\frac{(n-1)\pi}{4(n+1)}}\big)\csc\big({\frac{\pi}{2(n+1)}}\big)\cos\big({\frac{\pi}{4}}\big)$} & \quad\text{if $n$ is odd,}\\
			\text{$4\sin\big({\frac{n\pi}{4(n+1)}}\big)\csc\big({\frac{(n+2)\pi}{4(n+1)}}\big)\cos\big({\frac{\pi}{4}}\big)$} & \quad\text{if $n$ is even;}
		\end{cases}\\
		=&~\begin{cases}
			\text{$ 2\Big(\cos\big({\frac{\pi}{2(n+1)}}\big)-\sin \big({\frac{\pi}{2(n+1)}}\big)\Big)\csc\big({\frac{\pi}{2(n+1)}}\big)$} & \quad\text{if $n$ is odd,}\\
			\text{$2\Big(\sin \big(\frac{\pi}{2}\big)-\sin \big({\frac{\pi}{2(n+1)}}\big)\Big)\csc\big({\frac{\pi}{2(n+1)}}\big)$} & \quad\text{if $n$ is even;}
		\end{cases}\\
		=&~\begin{cases}
			\text{$ 2\Big(\cot\big({\frac{\pi}{2(n+1)}}\big)-1\Big)$} & \quad\text{if $n$ is odd,}\\
			\text{$2\Big(\csc\big({\frac{\pi}{2(n+1)}}\big)-1\Big)$} & \quad\text{if $n$ is even.}
		\end{cases} 
	\end{align*}
\end{proof}

Now, we collect some results related to the spectral radius, second-largest eigenvalue, and energy for the adjacency matrices of trees. 

\begin{thm}[\cite{lovasz1973eigenvalues}]\label{spectral-min}
	Let $T$ be a tree on $n$ vertices. Then $\lambda_1(T)\geq 2\cos{\frac{\pi}{n+1}}$ with equality if and only if $T\cong P_n$.
\end{thm} 

\begin{thm}[{\cite[Theorem 2]{hofmeister1997two}}]\label{spectral-max}
	Let $T$ be a tree on $n\geq 4$ vertices and $T\ncong K_{1,n-1}$. Then $\lambda_1(T)\leq \sqrt{\frac{n-1+\sqrt{n^2-6n+13}}{2}}$ with equality if and only if $T\cong T_{n,3}^{0,n-4}$.
\end{thm} 

\begin{thm}[{\cite[Theorem 2]{yuan1989sharp}}]\label{sec-larg-min}
	Let $T$ be a tree of order $n\geq 4$ and $T\ncong K_{1,n-1}, T_{n,3}^{0,n-4}$. Then $\lambda_2(T)\geq 1$.
\end{thm}

\begin{thm}[{\cite[Theorem 10]{hofmeister1997two}}]\label{sec-larg-max}
	Let $T$ be a tree with $n\geq 4$ vertices. 
	\begin{enumerate}
		\item If $T \ncong T_{n,5}^{s-1,s-1}$, then $\lambda_2(T)\leq \sqrt{\frac{n-3}{2}}$. The equality holds if and only if $n=2s+3$ and $T\cong T_{n,4}^{s-1,s-1}, T_{n,5}^{s-2,s-1},T_{n,6}^{s-2,s-2}$.
		\item If $T \cong T_{n,5}^{s-1,s-1}$, then $\lambda_2(T)=x>\sqrt{\frac{n-3}{2}}$, where $x$ is the positive root of the equation $x^3+x^2-(s+1)-s=0$. 
	\end{enumerate}
\end{thm}

\begin{lem}[{\cite[Proposition 4]{gutman1977acyclic}}]\label{energy-ordering}
	Let $T$ be a tree on $n\geq 2$ vertices such that $T$ is not isomorphic to $K_{1,n-1},T_{n,3}^{0,n-4},T_{n,3}^{1,n-3}$ and $T_{n,4}^{0,n-5}$. Then
	$$E_A(T)>E_A(T_{n,4}^{0,n-5})>E_A(T_{n,3}^{1,n-3})>E_A(T_{n,3}^{0,n-4})>E_A(K_{1,n-1}).$$
\end{lem}

\begin{lem}[{\cite[Proposition 3]{gutman1977acyclic}}]\label{max-energy}
	Let $T$ be a tree on $n\geq 2$ vertices such that $T\ncong K_{1,n-1}, P_n$. Then 
	$$E_A(K_{1,n-1})<E_A(T)<E_A(P_n).$$
\end{lem}

Next, we state a result about the minimum second largest $\mathcal{E}$-eigenvalue of trees.

\begin{thm}[\cite{mahato2022minimizers}]\label{second-larg}
	Let $T$ be a tree on $n\geq 4$ vertices other than the star. Then 
	$$\xi_2(T)\geq \sqrt{\frac{13n-35-\sqrt{169n^2-974n+1417}}{2}} $$
	with equality if and only if $T\cong T_{n,3}^{0,n-4}$.
\end{thm}

The following theorem is about the characterization of trees with minimum $\mathcal{E}$-energy.

\begin{thm}[\cite{mahato2022minimizers}]\label{min-energy}
	Let $T$ be a tree on $n\geq 5$ vertices. Then 
	$$E_{\mathcal{E}}(T)\geq 2\sqrt{13n-35+8\sqrt{n-3}}$$
	with equality if and only if $T\cong T_{n,3}^{0,n-4}$. 
\end{thm}

\section{Extremal problems for $\mathcal{E}$-spectral radius and least $\mathcal{E}$-eigenvalue}

In this section, we characterize the graphs with the minimum and maximum $\mathcal{E}$-spectral radius among the complements of all trees on $n$ vertices. As a consequence, we determine the graphs for which the least $\mathcal{E}$-eigenvalues attain the minimum and maximum value among $\mathcal{T}_n^c$, where $\mathcal{T}_n^c$ denote the complements of all trees on $n$ vertices. First, we give an ordering of the complements of trees with diameter $3$ according to their $\mathcal{E}$-spectral radius.

\begin{thm}\label{comple-spectral-ordering}
	Let $T$ be a tree with diameter $3$, that is, $T\cong T_{n,3}^{a,b}$ with $a+b=n-4$, $b\geq a\geq 0$. Then the complements of $T_{n,3}^{a,b}$ can be ordered with respect to their $\mathcal{E}$-spectral radius as follows:
	$$\xi_1\big((T_{n,3}^{0,n-4})^c\big)> \xi_1\big((T_{n,3}^{1,n-3})^c\big)>\hdots > \xi_1\Big(\big(T_{n,3}^{\lfloor \frac{n-4}{2}\rfloor,\lceil \frac{n-4}{2}\rceil}\big)^c\Big). $$
\end{thm}
\begin{proof}
	By Lemma \ref{spec-tpq} , we have
	\begin{align*}
		\xi_1\big((T_{n,3}^{a,b})^c\big)=&~ \sqrt{\frac{4n+1+\sqrt{(4n+1)^2-64(a+1)(b+1)}}{2}}\\
		=&~ \sqrt{\frac{4n+1+\sqrt{(4n+1)^2-64(n-3+a(n-4-a))}}{2}} \qquad \text{(since, $a+b=n-4$)}.
	\end{align*}
	
	Now, consider the function $f(x)=4n+1+\sqrt{(4n+1)^2-64(n-3+x(n-4-x))}$, where $0\leq x \leq \lfloor \frac{n-4}{2}\rfloor$. Therefore, $$f^{\prime}(x)= \frac{-32(n-4-2x)}{\sqrt{(4n+1)^2-64(n-3+x(n-4-x))}}\leq 0.$$
	Hence, $f(x)$ is an decreasing function of $x$, where $0\leq x \leq \lfloor \frac{n-4}{2}\rfloor$. Therefore, $\xi_1\big((T_{n,3}^{a,b})^c\big)=\sqrt{\frac{f(a)}{2}}$ is an decreasing function for $0\leq a \leq \lfloor \frac{n-4}{2}\rfloor$. This completes the proof.
\end{proof}

As a consequence of the above result, we get the following corollaries.

\begin{cor}\label{max-spectral-diam3}
	Let $T$ be a tree on $n\geq 4$ vertices with $\diam(T)=3$. Then 
	$$\xi_1(T^c)\leq \sqrt{\frac{4n+1+\sqrt{(4n+1)^2-64(n-3)}}{2}}$$
	with equality if and only if $T\cong T_{n,3}^{0,n-4}$.
\end{cor}
\begin{proof}
	The proof follows from Lemma \ref{spec-tpq} and Theorem \ref{comple-spectral-ordering}.
\end{proof}

\begin{cor}\label{min-spectral-diam3}
	If $T$ is a tree with diameter $3$, then
	\begin{align*}
		\xi_1(T^c)\geq \begin{cases}
			\text{$ \sqrt{\frac{4n+1+\sqrt{72n-63}}{2}}$} & \quad\text{if $n$ is even,}\\
			\text{$ \sqrt{\frac{4n+1+\sqrt{72n-47}}{2}}$} & \quad\text{if $n$ is odd.}
		\end{cases} 
	\end{align*}
	The equality holds if and only if $T\cong T_{n,3}^{\lfloor \frac{n-4}{2}\rfloor,\lceil \frac{n-4}{2}\rceil}$.
\end{cor}
\begin{proof}
	If $T\cong T_{n,3}^{\lfloor \frac{n-4}{2}\rfloor,\lceil \frac{n-4}{2}\rceil}$, by Lemma \ref{spec-tpq} it follows that
	\begin{align*}
		\xi_1(T^c)=&~ \sqrt{\frac{4n+1+\sqrt{(4n+1)^2-64\big(\lceil \frac{n-2}{2}\rceil \lfloor \frac{n-2}{2}\rfloor\big)}}{2}}\\
		=&~ \begin{cases}
			\text{$ \sqrt{\frac{4n+1+\sqrt{72n-63}}{2}}$} & \quad\text{if $n$ is even,}\\
			\text{$ \sqrt{\frac{4n+1+\sqrt{72n-47}}{2}}$} & \quad\text{if $n$ is odd.}
		\end{cases}
	\end{align*}
	If $T\ncong T_{n,3}^{\lfloor \frac{n-4}{2}\rfloor,\lceil \frac{n-4}{2}\rceil}$, then, by Theorem \ref{comple-spectral-ordering}, it follows that $\xi_1(T^c)>\xi_1\Big(\big(T_{n,3}^{\lfloor \frac{n-4}{2}\rfloor,\lceil \frac{n-4}{2}\rceil}\big)^c\Big)$. This completes the proof.
\end{proof}

In the following theorem, we characterize the trees whose complement has the minimum $\mathcal{E}$-spectral radius among the complements of all trees on $n$ vertices. Note that if $T$ is a tree with $\diam(T)>3$, then $\mathcal{E}(T^c)=2A(T)$. If $\diam(T)=3$, then $\mathcal{E}(T^c)\geq 2A(T)$ entrywise.

\begin{thm}\label{spec-min}
	Let $T$ be a tree of order $n\geq 4$. Then 
	$$\xi_1(T^c)\geq \xi_1(P_n^c) $$
	with equality if and only if $T\cong P_n$.
\end{thm}
\begin{proof}
	Since $P_4$ is the only tree on $4$ vertices with connected complement, we assume that $n\geq 5$. For $n\geq 5$, we have $\xi_1(P_n^c)=2\lambda_1(P_n)= 4\cos{\frac{\pi}{n+1}}<4$.
	
	If $T$ is a tree with diameter $3$, by Corollary \ref{min-spectral-diam3} it follows that 
	$$\xi_1(T^c)\geq \begin{cases}
		\text{$ \sqrt{\frac{4n+1+\sqrt{72n-63}}{2}}$} & \quad\text{if $n$ is even,}\\
		\text{$ \sqrt{\frac{4n+1+\sqrt{72n-47}}{2}}$} & \quad\text{if $n$ is odd.}
	\end{cases} $$
	It is easy to check that $\sqrt{\frac{4n+1+\sqrt{72n-47}}{2}}>\sqrt{\frac{4n+1+\sqrt{72n-63}}{2}}>4$. Therefore, $\xi_1(T^c)\geq \xi_1(P_n^c)$.
	
	If $T$ is a tree with $\diam(T)\geq 4$, then the proof follows from Theorem \ref{spec-min}.
\end{proof}

Now, we determine the unique tree whose complement has maximum $\mathcal{E}$-spectral radius in $\mathcal{T}_n^c$.

\begin{thm}\label{spec-max}
	Let $T$ be a tree of order $n\geq 4$. Then 
	$$\xi_1(T^c)\leq \xi_1\big((T_{n,3}^{0,n-4})^c\big) $$
	with equality if and only if $T\cong T_{n,3}^{0,n-4}$.
\end{thm}
\begin{proof}
	If $T$ is a tree with diameter $3$, then the proof follows from Corollary \ref{max-spectral-diam3}. If $T$ is a tree with $\diam(T)\geq 4$, then  $\xi_1(T^c)=2\lambda_1(T)$ and the proof follows from Theorem \ref{spectral-max}.  
\end{proof}

Now, we consider the extremal problem for the least $\mathcal{E}$-eigenvalue of complements of trees and characterize the extremal graphs. First, we show that the $\mathcal{E}$-eigenvalues of the complements of trees are symmetric about the origin.

\begin{thm}\label{spec-sym}
	Let $T$ be a tree of order $n\geq 4$. Then the eigenvalues of $\mathcal{E}(T^c)$ are symmetric about the origin, that is, if $\lambda$ is an eigenvalue of $\mathcal{E}(T^c)$ with multiplicity $k$, then $-\lambda$ is also an eigenvalue of $\mathcal{E}(T^c)$ with multiplicity $k$.
\end{thm}
\begin{proof}
	If $T$ is a tree with diameter $3$, then the proof follows from Lemma \ref{spec-tpq}. Let $T$ be a tree with $\diam(T)\geq 4$. Then $\mathcal{E}(T^c)=2A(T)$. Since every tree is a bipartite graph, the eigenvalues of $A(T)$ are symmetric about the origin, and hence the $\mathcal{E}$-eigenvalues of $T^c$ are also symmetric about the origin.
\end{proof}

Let $\xi_n(T^c)$ denote the least $\mathcal{E}$-eigenvalue of $\mathcal{E}(T^c)$. In the following theorems, we characterize the trees whose complements have minimum and maximum least $\mathcal{E}$-eigenvalue in $\mathcal{T}_n^c$, respectively.

\begin{thm}
	Let $T$ be a tree of order $n\geq 4$. Then 
	$$\xi_n(T^c)\geq \xi_n\big((T_{n,3}^{0,n-4})^c\big) $$
	with equality if and only if $T\cong T_{n,3}^{0,n-4}$.
\end{thm}
\begin{proof}
	The proof follows from Theorem \ref{spec-max} and Theorem \ref{spec-sym}.
\end{proof}

\begin{thm}
	Let $T$ be a tree of order $n\geq 4$. Then 
	$$\xi_n(T^c)\leq \xi_n(P_n^c) $$
	with equality if and only if $T\cong P_n$.
\end{thm}
\begin{proof}
	The proof follows from Theorem \ref{spec-min} and Theorem \ref{spec-sym}.
\end{proof}

\section{Extremal problems for the second largest $\mathcal{E}$-eigenvalue}

In this section, we study the extremal problems for the second largest $\mathcal{E}$-eigenvalue of complements of trees and characterize the extremal graphs. First, we give an ordering of the complements of trees with diameter $3$ according to their second-largest $\mathcal{E}$-eigenvalues.

\begin{thm}\label{comple-sec-ordering}
	Let $T$ be a tree with diameter $3$, that is, $T\cong T_{n,3}^{a,b}$ with $a+b=n-4$, $b\geq a\geq 0$. Then the complements of the trees $T_{n,3}^{a,b}$ can be ordered with respect to their second largest $\mathcal{E}$-eigenvalues as follows:
	$$\xi_2\big((T_{n,3}^{0,n-4})^c\big)< \xi_2\big((T_{n,3}^{1,n-3})^c\big)<\hdots < \xi_2\Big(\big(T_{n,3}^{\lfloor \frac{n-2}{2}\rfloor,\lceil \frac{n-2}{2}\rceil}\big)^c\Big). $$
\end{thm}
\begin{proof}
	By Lemma \ref{spec-tpq}, it follows  that
	\begin{align*}
		\xi_2\big((T_{n,3}^{a,b})^c\big)=&~ \sqrt{\frac{4n+1-\sqrt{(4n+1)^2-64(a+1)(b+1)}}{2}}\\
		=&~ \sqrt{\frac{4n+1-\sqrt{(4n+1)^2-64(n-3+a(n-4-a))}}{2}} \qquad \text{(since, $a+b=n-4$)}.
	\end{align*}
	
	Now, consider the function $f(x)=4n+1-\sqrt{(4n+1)^2-64(n-3+x(n-4-x))}$, where $0\leq x \leq \lfloor \frac{n-4}{2}\rfloor$. Therefore, $$f^{\prime}(x)= \frac{32(n-4-2x)}{\sqrt{(4n+1)^2-64(n-3+x(n-4-x))}}\geq 0.$$
	Hence, $f(x)$ is an increasing function of $x$ for $0\leq x \leq \lfloor \frac{n-4}{2}\rfloor$. Therefore, $\xi_2\big((T_{n,3}^{a,b})^c\big)=\sqrt{\frac{f(a)}{2}}$ is an increasing function for $0\leq a \leq \lfloor \frac{n-4}{2}\rfloor$. This completes the proof.
\end{proof}

As a consequence of the above result, we get the following corollaries.

\begin{cor}\label{min-second-diam3}
	If $T$ is a tree with diameter $3$, then
	$$\xi_2(T^c)\geq 2\sqrt{4n+1-\sqrt{(4n+1)^2-64(n-3)}} $$
	with equality if and only if $T\cong T_{n,3}^{0,n-4}$.
\end{cor}
\begin{proof}
	The proof follows from Lemma \ref{spec-tpq} and Theorem \ref{comple-sec-ordering}.
\end{proof}

\begin{cor}\label{max-second-diam3}
	If $T$ is a tree with diameter $3$, then
	\begin{align*}
		\xi_2(T^c)\leq \begin{cases}
			\text{$ \sqrt{\frac{4n+1-\sqrt{72n-63}}{2}}$} & \quad\text{if $n$ is even,}\\
			\text{$ \sqrt{\frac{4n+1-\sqrt{72n-47}}{2}}$} & \quad\text{if $n$ is odd.}
		\end{cases} 
	\end{align*}
	The equality holds if and only if $T\cong T_{n,3}^{\lfloor \frac{n-4}{2}\rfloor,\lceil \frac{n-4}{2}\rceil}$.
\end{cor}
\begin{proof}
	If $T\cong T_{n,3}^{\lfloor \frac{n-4}{2}\rfloor,\lceil \frac{n-4}{2}\rceil}$, by Lemma \ref{spec-tpq} it follows that
	\begin{align*}
		\xi_2(T^c)=&~ \sqrt{\frac{4n+1-\sqrt{(4n+1)^2-64\big(\lceil \frac{n-2}{2}\rceil \lfloor \frac{n-2}{2}\rfloor\big)}}{2}}\\
		=&~ \begin{cases}
			\text{$ \sqrt{\frac{4n+1-\sqrt{72n-63}}{2}}$} & \quad\text{if $n$ is even,}\\
			\text{$ \sqrt{\frac{4n+1-\sqrt{72n-47}}{2}}$} & \quad\text{if $n$ is odd.}
		\end{cases}
	\end{align*}
	If $T\ncong T_{n,3}^{\lfloor \frac{n-4}{2}\rfloor,\lceil \frac{n-4}{2}\rceil}$, by Theorem \ref{comple-sec-ordering} it follows that $\xi_2(T^c)<\xi_2\Big(\big(T_{n,3}^{\lfloor \frac{n-2}{2}\rfloor,\lceil \frac{n-2}{2}\rceil}\big)^c\Big)$. This completes the proof.
\end{proof}

In the following theorem, we determine the unique tree whose complement has minimum second largest $\mathcal{E}$-eigenvalue in $\mathcal{T}_n^c$.

\begin{thm}\label{comple-sec-larg-min}
	Let $T$ be a tree of order $n\geq 4$. Then 
	$$\xi_2(T^c)\geq \xi_2\big((T_{n,3}^{0,n-4})^c\big) $$
	with equality if and only if $T\cong T_{n,3}^{0,n-4}$.
\end{thm}
\begin{proof}
	If $T$ is a tree on $n\geq 4$ vertices with diameter $3$, then the proof follows from Corollary \ref{min-second-diam3}.
	
	Let $T$ be a tree with $\diam(T)\geq 4$. Then $\mathcal{E}(T^c)=2A(T)$, and hence by Theorem \ref{sec-larg-min}, we have $\xi_2(T^c)\geq 2$. Note that $\sqrt{(4n-7)^2+144}>4n-7$ and hence $(4n+1-\sqrt{(4n-7)^2+144})<8$. Therefore, $\xi_2\big((T_{n,3}^{0,n-4})^c\big)=\sqrt{\frac{4n+1-\sqrt{(4n-7)^2+144}}{2}}<2$. Thus, $\xi_2(T^c)>\xi_2\big((T_{n,3}^{0,n-4})^c\big)$. This completes the proof. 
\end{proof}

Next, we characterize the maximal graphs for the second largest $\mathcal{E}$-eigenvalue of complements of trees.

\begin{thm}\label{comple-sec-larg-max}
	Let $T$ be a tree with $n\geq 4$ vertices. 
	\begin{enumerate}
		\item If $T \ncong T_{n,5}^{s-1,s-1}$, then $\xi_2(T)\leq \sqrt{2(n-3)}$. The equality holds if and only if $n=2s+3$ and $T\cong T_{n,4}^{s-1,s-1}, T_{n,5}^{s-2,s-1},T_{n,6}^{s-2,s-2}$.
		\item If $T \cong T_{n,5}^{s-1,s-1}$, then $\xi_2(T)=2x>\sqrt{2(n-3)}$, where $x$ is the positive root of the equation $x^3+x^2-(s+1)-s=0$. 
	\end{enumerate}
\end{thm}
\begin{proof}
	If $T$ is a tree on $n\geq 4$ vertices with diameter $3$, by Corollary \ref{max-second-diam3} it follows that  
	\begin{align*}
		\xi_2(T^c)\leq \begin{cases}
			\text{$ \sqrt{\frac{4n+1-\sqrt{72n-63}}{2}}$} & \quad\text{if $n$ is even,}\\
			\text{$ \sqrt{\frac{4n+1-\sqrt{72n-47}}{2}}$} & \quad\text{if $n$ is odd.}
		\end{cases} 
	\end{align*}
	It is easy to check that $\sqrt{\frac{4n+1-\sqrt{72n-47}}{2}}<\sqrt{\frac{4n+1-\sqrt{72n-63}}{2}}<\sqrt{2(n-3)}$ for $n\geq 4$. Therefore, $\xi_2(T^c)<\sqrt{2(n-3)}$.
	
	Let $T$ be a tree with $\diam(T)\geq 4$. Then  $\xi_2(T^c)=2\lambda_2(A(T))$ and the proof follows from Theorem \ref{sec-larg-max}.
\end{proof}

Now, we give a Nordhaus-Gaddum type lower bound for the second largest $\mathcal{E}$-eigenvalue of a tree and its complement, which directly follows from Theorem \ref{second-larg} and Theorem \ref{comple-sec-larg-min}.

\begin{thm}
	Let $T$ be a tree of order $n\geq 4$. Then 
	$$\xi_2(T)+\xi_2(T^c)\geq \Bigg(\sqrt{\frac{13n-35-\sqrt{169n^2-974n+1417}}{2}}+ \sqrt{\frac{4n+1-\sqrt{16n^2-56n+193}}{2}}\Bigg)$$
	with equality if and only if $T\cong T_{n,3}^{0,n-4}$.
\end{thm}

\section{Extremal problems for $\mathcal{E}$-energy}
In this section, we characterize the complements of trees with the minimum and maximum $\mathcal{E}$-energy among the complements of all trees on $n$ vertices, respectively. To begin with, in the following lemma, we give an ordering of the complements of trees with diameter $3$ according to their $\mathcal{E}$-energy.

\begin{thm}\label{comple-energy-ordering}
	Let $T$ be a tree with diameter $3$, that is, $T\cong T_{n,3}^{a,b}$ with $a+b=n-4$, $b\geq a\geq 0$. Then the complements of $T_{n,3}^{a,b}$ can be ordered with respect to their $\mathcal{E}$-energy as follows: $$E_{\mathcal{E}}\big((T_{n,3}^{0,n-4})^c\big)< E_{\mathcal{E}}\big((T_{n,3}^{1,n-3})^c\big)<\hdots < E_{\mathcal{E}}\big((T_{n,3}^{\lfloor \frac{n-4}{2}\rfloor,\lceil \frac{n-4}{2}\rceil})^c\big). $$
\end{thm}
\begin{proof}
	By Lemma \ref{spec-tpq}, it follows that $E_{\mathcal{E}}\big((T_{n,3}^{a,b})^c\big)=2\Big(\xi_1\big((T_{n,3}^{a,b})^c\big)+\xi_2\big((T_{n,3}^{a,b})^c\big)\Big)$, where
	\begin{align*}
		\xi_1\big((T_{n,3}^{a,b})^c\big)=&~ \sqrt{\frac{4n+1+\sqrt{(4n+1)^2-64(a+1)(b+1)}}{2}} \quad \text{and}\\  \xi_2\big((T_{n,3}^{a,b})^c\big)=&~ \sqrt{\frac{4n+1-\sqrt{(4n+1)^2-64(a+1)(b+1)}}{2}}.
	\end{align*}
	Therefore, 
	\begin{align*}
		E_{\mathcal{E}}\big((T_{n,3}^{a,b})^c\big)=&~ 2\sqrt{4n+1+8\sqrt{(a+1)(b+1)}}\\  
		=&~ 2\sqrt{4n+1+8\sqrt{n-3+a(n-4-a)}} \quad (\text{since,~ $a+b=n-4$}).
	\end{align*}
	
	Now, consider the function $f(x)=4n+1+8\sqrt{n-3+x(n-4-x)}$, where $0\leq x \leq \lfloor \frac{n-4}{2}\rfloor$. Therefore, 
	$$f^{\prime}(x)= \frac{8(n-4-2x)}{4n+1+8\sqrt{n-3+x(n-4-x)}}\geq 0.$$
	Hence, $f(x)$ is an increasing function of $x$, where $0\leq x \leq \lfloor \frac{n-4}{2}\rfloor$. Therefore, $E_{\mathcal{E}}\big((T_{n,3}^{a,b})^c\big)=2\sqrt{f(a)}$ is an increasing function for $0\leq a \leq \lfloor \frac{n-4}{2}\rfloor$. This completes the proof.
\end{proof}

As a consequence of the above result, we get the following corollaries.

\begin{cor}\label{min-ener-diam3}
	If $T$ is a tree with diameter $3$, then
	$$E_{\mathcal{E}}(T^c)\geq 2\sqrt{4n+1+8\sqrt{n-3}} $$
	with equality if and only if $T\cong T_{n,3}^{0,n-4}$.
\end{cor}
\begin{proof}
	By Lemma \ref{spec-tpq}, it follows that 
	\begin{align*}
		\xi_1\big((T_{n,3}^{0,n-4})^c\big)=&~ \sqrt{\frac{4n+1+\sqrt{(4n+1)^2-64(n-3)}}{2}}, \quad \text{and}\\
		\xi_2\big((T_{n,3}^{0,n-4})^c\big)=&~ \sqrt{\frac{4n+1-\sqrt{(4n+1)^2-64(n-3)}}{2}}. 
	\end{align*}
	Thus, 
	\begin{align*}
		E_{\mathcal{E}}\big((T_{n,3}^{0,n-4})^c\big)=&~ 2\Big(\xi_1\big((T_{n,3}^{0,n-4})^c\big)+\xi_2\big((T_{n,3}^{0,n-4})^c\big)\Big)=2\sqrt{4n+1+8\sqrt{n-3}}.
	\end{align*}
	Now, the proof follows from Theorem \ref{comple-energy-ordering}.
\end{proof}

\begin{cor}\label{max-ener-diam3}
	If $T$ is a tree with diameter $3$, then
	\begin{align*}
		E_{\mathcal{E}}(T^c)\leq \begin{cases}
			\text{$2\sqrt{8n-7}$} & \quad\text{if $n$ is even,}\\
			\text{$2\sqrt{4n+1+4\sqrt{n^2-4n+3}}$} & \quad\text{if $n$ is odd.}
		\end{cases}    
	\end{align*}
	The equality holds if and only if $T\cong T_{n,3}^{\lfloor \frac{n-4}{2}\rfloor,\lceil \frac{n-4}{2}\rceil}$.
\end{cor}
\begin{proof}
	It follows from Lemma \ref{spec-tpq} that
	\begin{align*}
		\xi_1\big((T_{n,3}^{\lfloor \frac{n-4}{2}\rfloor,\lceil \frac{n-4}{2}\rceil})^c\big)=&~ \sqrt{\frac{4n+1+\sqrt{(4n+1)^2-64\big(\lceil \frac{n-2}{2}\rceil \lfloor \frac{n-2}{2}\rfloor\big)}}{2}}\\
		=&~ \begin{cases}
			\text{$ \sqrt{\frac{4n+1+\sqrt{72n-63}}{2}}$} & \quad\text{if $n$ is even,}\\
			\text{$ \sqrt{\frac{4n+1+\sqrt{72n-47}}{2}}$} & \quad\text{if $n$ is odd.}
		\end{cases}
	\end{align*}
	and
	\begin{align*}
		\xi_2\big((T_{n,3}^{\lfloor \frac{n-4}{2}\rfloor,\lceil \frac{n-4}{2}\rceil})^c\big)=&~ \sqrt{\frac{4n+1-\sqrt{(4n+1)^2-64\big(\lceil \frac{n-2}{2}\rceil \lfloor \frac{n-2}{2}\rfloor\big)}}{2}}\\
		=&~ \begin{cases}
			\text{$ \sqrt{\frac{4n+1-\sqrt{72n-63}}{2}}$} & \quad\text{if $n$ is even,}\\
			\text{$ \sqrt{\frac{4n+1-\sqrt{72n-47}}{2}}$} & \quad\text{if $n$ is odd.}
		\end{cases}
	\end{align*}
	Therefore, 
	\begin{align*}
		E_{\mathcal{E}}\big((T_{n,3}^{\lfloor \frac{n-4}{2}\rfloor,\lceil \frac{n-4}{2}\rceil})^c\big)=&~ 2\Big(\xi_1\big((T_{n,3}^{\lfloor \frac{n-4}{2}\rfloor,\lceil \frac{n-4}{2}\rceil})^c\big)+\xi_2\big((T_{n,3}^{\lfloor \frac{n-4}{2}\rfloor,\lceil \frac{n-4}{2}\rceil})^c\big)\Big)\\
		=&~ \begin{cases}
			\text{$2\sqrt{8n-7}$} & \quad\text{if $n$ is even,}\\
			\text{$2\sqrt{4n+1+4\sqrt{n^2-4n+3}}$} & \quad\text{if $n$ is odd.}
		\end{cases}
	\end{align*}
	Now, the proof follows from Theorem \ref{comple-energy-ordering}.
\end{proof}

Next, we find the $\mathcal{E}$-energy of $(T_{n,4}^{0,n-5})^c$ by computing the $\mathcal{E}$-spectrum of $(T_{n,4}^{0,n-5})^c$. This will be used in the proof of Theorem \ref{complement-min-energy}. 

\begin{figure}[h!]
	\centering
	\includegraphics[scale= 0.70]{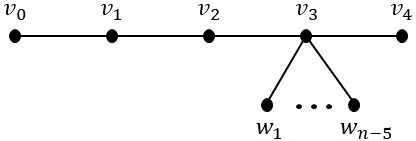}
	\caption{The tree $T_{n,4}^{0,n-5}$.}
	\label{fig3}
\end{figure}

\begin{lem}\label{comple-energy-411}
	For $n\geq 5$, the $\mathcal{E}$-energy of $E_{\mathcal{E}}\big((T_{n,4}^{0,n-5})^c\big)$ is given by 
	$$E_{\mathcal{E}}\big((T_{n,4}^{0,n-5})^c\big)=2\sqrt{4(n-1)+8\sqrt{2n-7}}.$$
\end{lem}
\begin{proof}
	The eccentricity matrix of $(T_{n,4}^{0,n-5})^c$ is given by
	\[ \mathcal{E}\big((T_{n,4}^{0,n-5})^c\big)=
	\begin{blockarray}{ccccccccc}
		& v_0 & v_1 & v_2 & v_3 & v_4 & w_1 & \hdots & w_{n-5} \\
		\begin{block}{c(cccccccc)}
			v_0 & 0 & 2 & 0 & 0 & 0 & 0 & \hdots & 0 \\
			v_1 & 2 & 0 & 2 & 0 & 0 & 0 & \hdots & 0 \\
			v_2 & 0 & 2 & 0 & 2 & 0 & 0 & \hdots & 0 \\
			v_3 & 0 & 0 & 2 & 0 & 2 & 2 & \hdots & 2 \\
			v_4 & 0 & 0 & 0 & 2 & 0 & 0 & \hdots & 0 \\
			w_1 & 0 & 0 & 0 & 2 & 0 & 0 & \hdots & 0 \\
			\vdots & \vdots & \vdots & \vdots & \vdots & \vdots & \vdots & \ddots & \vdots \\
			w_{n-5} & 0 & 0 & 0 & 2 & 0 & 0 & \hdots & 0 \\
		\end{block}
	\end{blockarray}~.\]
	
	It is easy to see that the rank of $\mathcal{E}\big((T_{n,4}^{0,n-5})^c\big)$ is $4$, and hence $0$ is an eigenvalue of $\mathcal{E}\big((T_{n,4}^{0,n-5})^c\big)$ with multiplicity $n-4$. 
	
	If $W=\{v_4,w_1,\hdots,w_{n-5}\}$, then $\Pi_2=\{v_0\}\cup \{v_1\} \cup \{v_2\} \cup \{v_3\} \cup U$ is an equitable partition of $\mathcal{E}\big((T_{n,4}^{0,n-5})^c\big)$ with the quotient matrix
	
	\[Q_2=
	\left( {\begin{array}{ccccc}
			0 & 2 & 0 & 0 & 0 \\
			2 & 0 & 2 & 0 & 0\\
			0 & 2 & 0 & 2 & 0 \\
			0 & 0 & 2 & 0 & 2(n-4) \\
			0 & 0 & 0 & 2 & 0 \\
	\end{array} } \right).\]
	Now, the eigenvalues of $Q_2$ are
	\begin{align*}
		-\sqrt{2n-2\pm 2\sqrt{n^2-10n+29}}, \quad 0, \quad \text{and} \quad  \sqrt{2n-2\pm 2\sqrt{n^2-10n+29}}.    
	\end{align*}
	Therefore, by Theorem \ref{quo-spec}, we have 
	\[ \Spec_{\mathcal{E}}\big((T_{n,4}^{0,n-5})^c\big)=
	\left\{ {\begin{array}{ccc}
			-\sqrt{2n-2\pm 2\sqrt{n^2-10n+29}} & 0  & \sqrt{2n-2\pm 2\sqrt{n^2-10n+29}}\\
			1 & n-4 & 1 \\
	\end{array} } \right\}.\]
	
	Hence, $E_{\mathcal{E}}\big((T_{n,4}^{0,n-5})^c\big)=2\Big(\sqrt{2n-2+ 2\sqrt{n^2-10n+29}}+\sqrt{2n-2-2\sqrt{n^2-10n+29}}\Big)$. If $\alpha=\sqrt{2n-2+ 2\sqrt{n^2-10n+29}}$ and $\beta=\sqrt{2n-2-2\sqrt{n^2-10n+29}}$, then $\alpha^2+\beta^2=4n-4$ and $2\alpha \beta=8\sqrt{2n-7}$. Therefore,
	\begin{align*}
		E_{\mathcal{E}}\big((T_{n,4}^{0,n-5})^c\big)=&~2(\alpha+\beta)\\
		=&~2\Big(\sqrt{\alpha^2+\beta^2+2\alpha \beta}\Big)\\
		=&~2\bigg(\sqrt{4n-4+8\sqrt{2n-7}}\bigg).
	\end{align*}
\end{proof}

\begin{thm}\label{complement-min-energy}
	Let $T$ be a tree of order $n\geq 4$. Then 
	$$E_{\mathcal{E}}(T^c)\geq E_{\mathcal{E}}\big((T_{n,3}^{0,n-4})^c\big) $$
	with equality if and only if $T\cong T_{n,3}^{0,n-4}$.
\end{thm}
\begin{proof}
	For $n=4$, $T_{n,3}^{0,0}$ is the only tree with a connected complement. For $n=5$ and $6$, it is easy to see that  $E_{\mathcal{E}}(T^c)>E_{\mathcal{E}}\big((T_{n,3}^{0,n-4})^c\big)$ (see  Appendix). So, let us assume that $T$ is a tree on $n\geq 7$ vertices.
	
	If $T$ is a tree with $\diam(T)=3$, then the proof follows from Corollary \ref{min-ener-diam3}. If $T$ is a tree with $\diam(T)\geq 4$, then, by Lemma \ref{energy-ordering}, it follows that
	$$E_{\mathcal{E}}(T^c)=2E_A(T)\geq 2E_A(T_{n,4}^{0,n-5})=E_{\mathcal{E}}\big((T_{n,4}^{0,n-5})^c\big).$$
	Again, by Lemma \ref{comple-energy-411}, we have $E_{\mathcal{E}}\big((T_{n,4}^{0,n-5})^c\big)=2\sqrt{4(n-1)+8\sqrt{2n-7}}$. Note that $4(n-1)+8\sqrt{2n-7}>4n+1+8\sqrt{n-3}$ for $n\geq 7$ and hence $E_{\mathcal{E}}\big((T_{n,4}^{0,n-5})^c\big)>E_{\mathcal{E}}\big((T_{n,3}^{0,n-4})^c\big)$ for $n\geq 7$. This completes the proof.  
\end{proof}

In the following theorem, we characterize the complements of trees with maximum $\mathcal{E}$-energy.

\begin{thm}\label{comple-max-energy}
	Let $T$ be a tree of order $n\geq 4$. Then 
	$$E_{\mathcal{E}}(T^c)\leq E_{\mathcal{E}}(P_n^c) $$
	with equality if and only if $T\cong P_n$.
\end{thm}
\begin{proof}
	For $n=4$, $P_4$ is the only tree with a connected complement. For $n=5$ and $6$, it is easy to see that $E_{\mathcal{E}}(T^c)\leq E_{\mathcal{E}}(P_n^c)$ with equality if and only if $T\cong P_n$ (see, Appendix). So, let us assume that $T$ is a tree on $n\geq 7$ vertices. For $n\geq 7$, it follows from Theorem \ref{energy=path} that
	\begin{align*}
		E_{\mathcal{E}}(P_n^c)=2E_A(P_n)=&~\begin{cases}
			\text{$ 4\Big(\cot\big({\frac{\pi}{2(n+1)}}\big)-1\Big)$} & \quad\text{if $n$ is odd,}\\
			\text{$4\Big(\csc\big({\frac{\pi}{2(n+1)}}\big)-1\Big)$} & \quad\text{if $n$ is even;}
		\end{cases} \\
		\geq &~ \Big(\cot\big({\frac{\pi}{2(n+1)}}\big)-1\Big).
	\end{align*}
	
	We know that $\cot{x}>\big(\frac{1}{x}+\frac{1}{x-\pi}\big)$ for $0<x<\frac{\pi}{2}$. Since $0<\frac{\pi}{2(n+1)}<\frac{\pi}{2}$, therefore
	\begin{equation}\label{eqn1}
		E_{\mathcal{E}}(P_n^c)\geq 4\bigg(\cot\Big(\frac{\pi}{2(n+1)}\Big)-1\bigg)>4\Big(\frac{4n(n+1)}{(2n+1)\pi}-1\Big).    
	\end{equation}
	
	Let $T$ be a tree with $\diam(T)=3$. Therefore, by Corollary \ref{max-ener-diam3} it follows that
	\begin{align*}
		E_{\mathcal{E}}(T^c)\leq &~\begin{cases}
			\text{$2\sqrt{8n-7}$} & \quad\text{if $n$ is even,}\\
			\text{$2\sqrt{4n+1+4\sqrt{n^2-4n+3}}$} & \quad\text{if $n$ is odd;}
		\end{cases}\\
		\leq &~  2\sqrt{8n-7}.
	\end{align*}
	
	By using (\ref{eqn1}) one can check that 
	\begin{align*}
		E_{\mathcal{E}}(P_n^c)>4\Big(\frac{4n(n+1)}{(2n+1)\pi}-1\Big)>2\sqrt{8n-7}\geq E_{\mathcal{E}}(T^c) \quad \text{for $n\geq 7$}.   
	\end{align*}
	
	Thus, for any tree $T$ with diameter $3$, $E_{\mathcal{E}}(P_n^c)>E_{\mathcal{E}}(T^c)$. 
	
	Let $T$ be a tree with $\diam(T)\geq 4$. Therefore, $\mathcal{E}(T^c)=2A(T)$ and hence $E_{\mathcal{E}}(T^c)=2E_A(T)$. Now, by Lemma \ref{max-energy} it follows that
	$$E_{\mathcal{E}}(T^c)=2E_A(T)\leq 2E_A(P_n)= E_{\mathcal{E}}(P_n^c)$$
	and the equality holds if and only if $T\cong P_n$. This completes the proof.
\end{proof}

Finally, we obtain a Nordhaus-Gaddum type lower bound for the $\mathcal{E}$-energy of a tree and its complement, which directly follows from Theorem \ref{min-energy} and Theorem \ref{complement-min-energy}.

\begin{thm}
	Let $T$ be a tree of order $n\geq 4$. Then 
	$$E_{\mathcal{E}}(T)+E_{\mathcal{E}}(T^c)\geq 2\bigg(\sqrt{13n-35+8\sqrt{n-3}}+\sqrt{4n+1+8\sqrt{n-3}}\bigg)$$
	with equality if and only if $T\cong T_{n,3}^{0,n-4}$.
\end{thm}

\bibliographystyle{plain}
\bibliography{ecc-ref}

\section*{Appendix}\label{appendix}    

\begin{figure}[h!]
	\centering
	\includegraphics[scale= 0.60]{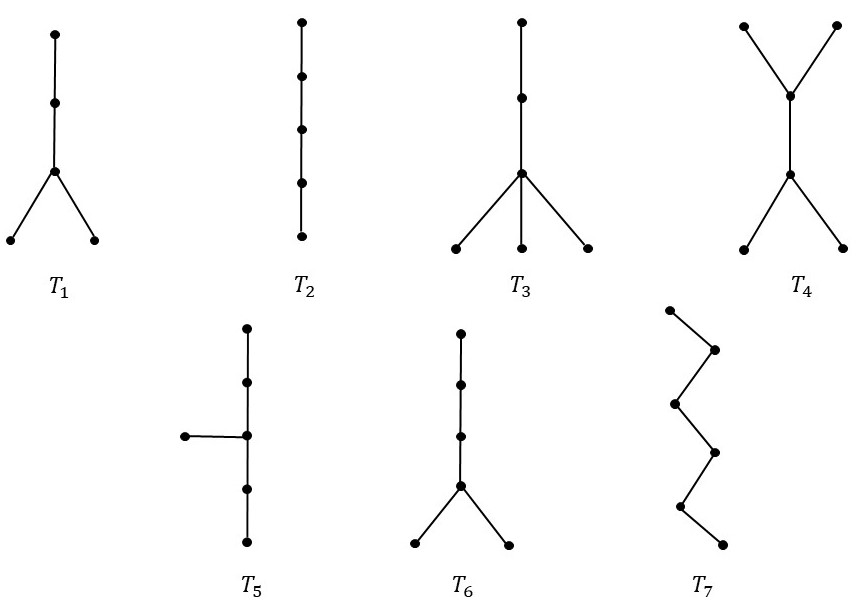}
	\caption{List of trees on $5$ and $6$ vertices with connected complements.}
\end{figure}

\begin{table}[h!]
	\centering
	\begin{tabular}{|c| c| c| c|}
		\hline
		Trees($T$) & $E_{\mathcal{E}}(T^c)$ & Trees($T$) & $E_{\mathcal{E}}(T^c)$ \\
		\hline
		$T_1$ & $\approx 10.4528$ & $T_5$ & $\approx 13.798$\\
		\hline
		$T_2$ & $\approx 10.9284$ & $T_6$ & $\approx 12.3108$\\
		\hline
		$T_3$ & $\approx 11.6372$ & $T_7$ & $\approx 13.9756$ \\
		\hline
		$T_4$ & $= 12$ &  &  \\
		\hline
	\end{tabular}
	\caption{$\mathcal{E}$-energy of complements of the trees $T_1$-$T_7$.  }
	\label{tab1}
\end{table}

\end{document}